\documentclass[11 pt, psamsfonts, leqno]{amsart}
\usepackage{amssymb}
\usepackage{amsmath}	
\usepackage{ bm}

 \calclayout
\def\ignore#1{\relax}

\ignore{
\usepackage[scaled=0.92]{helvet}	
\usepackage[subscriptcorrection,slantedGreek,nofontinfo,  mtphbi]{mtpro2}
}

\usepackage[plainpages=false, pdfpagelabels,  colorlinks=true, pdfstartview=FitV, linkcolor=blue, citecolor=blue, urlcolor=blue]{hyperref}

\usepackage{epic}
\usepackage{eepicemu}
\usepackage{graphicx}
\usepackage{epstopdf}

\newcommand\inlinegraphic[2][{scale=1.0}]{\begin{array}{c} \includegraphics[#1]{./EPS/#2}\end{array}}

\numberwithin{equation}{section}
\numberwithin{figure}{section}

\def\Z{{\mathbb Z}}

 \def\S{{\mathfrak{S}}}

\def\u #1 #2{\mathcal U(#1, #2)}  
\def\uhat #1 #2{\widehat{\mathcal U}(#1, #2)}  

\def\kt #1{{KT_{#1}}}  
\def\akt #1{\widehat{KT}_{#1}}  
\def\abmw #1{\widehat{W}_{#1}}  
\def\bmw #1{W_{#1}}  
\def\ahec #1{\widehat{H}_{#1}}  
\def\hec #1{H_{#1}}
\def\w #1 #2{\bmw {#1}^{(#2)}}  
\def\V #1 #2{V_{#1}^{(#2)}}  
\def\k #1 #2{ KT_{#1}^{(#2)}}
 
\def\p #1{\bm{#1}}
\def\pbar #1{\overline{\bm{ #1}}}

\def\inv{^{-1}}
\def\la{\lambda}

\def\p #1{\bm {#1}}
\def\pbar #1{\bar{\p #1}}

\def\mathbold{\bm}

\def\xbold{\bm x}
\def\boldx{\xbold}
\def\boldt{\bm t}
\def\tbold{\boldt}

\def\U{\mathbb U}

\def\End{{\rm End}}

\def\hods{\unskip\kern.55em\ignorespaces}

\theoremstyle{plain}
\newtheorem{theorem}{Theorem}[section]

\theoremstyle{plain}
\newtheorem{proposition}[theorem]{Proposition}

\theoremstyle{plain}

\theoremstyle{plain}
\newtheorem{lemma}[theorem]{Lemma}

\theoremstyle{definition}
\newtheorem{definition}[theorem]{Definition}
\theoremstyle{definition}
\newtheorem{example}[theorem]{Example}

\theoremstyle{definition}
\newtheorem{remark}[theorem]{Remark}

\theoremstyle{remark}

\title[Cyclotomic BMW algebras]{Cellularity of  Cyclotomic Birman--Wenzl--Murakami algebras }

\author{Frederick M. Goodman}
\address{Department of Mathematics\\ University of Iowa\\ Iowa
City, Iowa}
\email{ goodman@math.uiowa.edu}


\begin{document}
 \baselineskip=16pt 
 \maketitle

\dedicatory{Dedicated to Gus Lehrer on the occasion of his 60th birthday.}

\begin{abstract}    We show that  cyclotomic BMW algebras are cellular algebras.
\end{abstract}

\setcounter{tocdepth}{1}
\section{Introduction}
In this paper, we prove that the cyclotomic Birman--Wenzl--Murakami algebras are cellular, in the sense of Graham and Leherer ~\cite{Graham-Lehrer-cellular}.

The origin of the BMW algebras was in knot theory.  Shortly after the invention of the Jones link invariant ~\cite{jones-invariant}, Kauffman introduced   a new invariant of regular isotopy for  links in $S^3$, determined by certain skein relations ~\cite{Kauffman}.   
Birman and Wenzl ~\cite{Birman-Wenzl} and independently Murakami ~\cite{Murakami-BMW} then defined  a family braid group algebra quotients from  which Kauffman's invariant could be recovered.  These (BMW) algebras were  defined by generators and relations, but were implicitly modeled on certain algebras of tangles, whose definition was subsequently made explicit by Morton and Traczyk ~\cite{Morton-Traczyk}, as follows:
  Let $S$ be a commutative unital ring with invertible elements
$\rho$, $q$, and $\delta_0$ satisfying $\rho\inv - \rho = (q\inv -q) (\delta_0 - 1)$.  The {\em Kauffman tangle algebra}  $\kt{n, S}$  is the $S$--algebra of framed $(n, n)$--tangles in the disc cross the interval,  modulo Kauffman skein relations:
\begin{enumerate}
\item Crossing relation:
$
\quad \inlinegraphic[scale=.6]{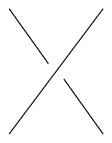} - \inlinegraphic[scale=.3]{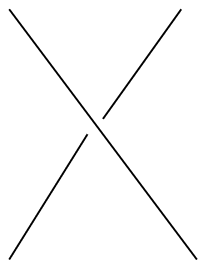} 
\quad = 
\quad
(q\inv - q)\,\left( \inlinegraphic[scale=1]{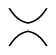} - 
\inlinegraphic[scale=1]{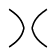}\right).
$
\item Untwisting relation:
$\quad 
\inlinegraphic{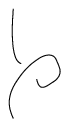} \quad = \quad \rho \quad
\inlinegraphic{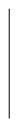} \quad\ \text{and} \quad\ 
\inlinegraphic{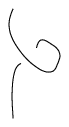} \quad = \quad \rho\inv \quad
\inlinegraphic{vertical_line}. 
$
\item  Free loop relation:  $T\, \cup \, \bigcirc = \delta_0 \, T. $
\end{enumerate}
Morton and Traczyk ~\cite{Morton-Traczyk}   showed that the $n$--strand algebra $\kt{n, S}$ is free of rank $(2n-1)!!$ as a module over $S$, and
Morton and Wassermann ~\cite{Morton-Wassermann} proved that the BMW algebras and the Kauffman tangle algebras are isomorphic.
 
It is natural to ``affinize" the BMW algebras to obtain BMW analogues of the affine Hecke   algebras of type $A$, see ~\cite{ariki-book}.  The affine Hecke algebra  can be realized geometrically as the algebra of braids in the annulus cross the interval,  modulo Hecke skein relations; this suggests defining the affine Kauffman tangle algebra as the algebra of framed $(n, n)$--tangles in the annulus cross the interval, modulo Kauffman skein relations.  However, Turaev  ~\cite{Turaev-Kauffman-skein} showed that the resulting algebra of  $(0,0)$--tangles is a  (commutative) polynomial algebra in infinitely many variables, so it  makes sense  to absorb this polynomial algebra into the ground ring.  (The ground ring gains infinitely many parameters corresponding to the generators of the polynomial algebra.)  One can also define a purely algebraic version of these algebras, by generators and relations ~\cite{H-O2}, the {\em affine BMW algebras}.  
 In ~\cite{GH1}, we showed that the two versions are isomorphic.

The affine BMW algebras have a distinguished generator $y_1$, which, in the geometric (Kauffman tangle) picture is represented by a braid with one strand wrapping around the hole in the annulus cross interval. {\em  Cyclotomic BMW algebras} are quotients 
 of the  affine BMW algebras   in which the generator $y_1$ satisfies a monic polynomial equation.    The affine and cyclotomic BMW algebras arise naturally in connection with knot theory in the solid torus, braid representations generated by $R$--matrices of symplectic and orthogonal quantum groups, and the representation theory of the ordinary BMW algebras (where the affine generators become Jucys--Murphy elements).  We refer the reader to ~\cite{GH2} for further discussion and references.

In order to get a good theory for cyclotomic BMW algebras, it is necessary to impose  conditions on the ground ring.  An appropriate condition, known as admissibility, was introduced by Wilcox and Yu in ~\cite{Wilcox-Yu}.  
Their condition has a simple formulation in terms of the 
 representation theory of the 2--strand cyclotomic BMW algebra,  and also translates    into explicit relations on the parameters.  
 
Let $\bmw {n, S, r}$  denote the cyclotomic quotient of the $n$--strand affine BMW algebra, in which the affine generator $y_1$ satisfies a polynomial relation of degree $r$, defined over a ring $S$ with appropriate parameters.
 It has been shown in ~\cite{GH2, GH3,  Wilcox-Yu2, Yu-thesis}  that if $S$ is   
 an admissible integral domain,  then 
 $\bmw{n, S, r}$  is a free $S$--module of rank $r^n (2n-1)!!$,  and is isomorphic to a cyclotomic version of the Kauffman tangle algebra.
  In this paper, we show that the techniques of ~\cite{GH2} can be modified to yield a cellular basis of the cyclotomic BMW algebras.

The cellularity of the ordinary BMW algebras has been shown by Xi ~\cite{Xi-BMW} and Enyang ~\cite{Enyang1, Enyang2}.  It is worth pointing out that if we specialize our proof for the cyclotomic case to the ordinary BMW algebras,  we end up showing that the tangle basis of ~\cite{Morton-Traczyk, Morton-Wassermann} is cellular;  in fact,  the proof would require only minor modifications of arguments already present in Morton--Wassermann ~\cite{Morton-Wassermann}.

Yu  ~\cite{Yu-thesis} has also shown that cyclotomic BMW algebras over admissible ground rings are cellular;  her result is slightly more general, since she used a broader definition of admissibility.  See also Remark \ref{remark:  on definition of cellularity}.

\section{Preliminaries}

\subsection{Definitions}

In the following, let
$S$ be a commutative unital ring containing   elements 
$\rho$, $q$, and  $\delta_j$, $j \ge 0$,   with $\rho$, $q$,  and $\delta_0$ invertible, satisfying the relation
$
\rho\inv - \rho=   (q\inv -q) (\delta_0 - 1).
$

\begin{definition}\label{definition affine Kauffman tangle algebra}

 The affine Kauffman tangle algebra $\akt {n, S, r}$ is the $S$--algebra of framed $(n, n)$--tangles in the annulus cross the interval, modulo Kauffman skein relations, namely the {\em crossing relation} and {\em untwisting relation}, as given in the introduction, and the {\em free loop relations}: 
 for $j \ge 0$,  $T\, \cup \, \varTheta_j =  \rho^{-j}\delta_j T, $
 where $T\, \cup\,\ \varTheta_j$ is the union
of an affine  tangle $T$ and a disjoint copy of the closed  curve $\varTheta_j$  that wraps 
$j$ times around the hole in the annulus cross the interval.
\end{definition}

\begin{figure}[ht]
$$
\inlinegraphic[scale=1.2]{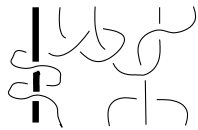}
$$
\caption{Affine $(4, 4)$--tangle diagram}\label{figure-affine tangle diagram}
\end{figure}

Affine tangles can be represented by {\em affine tangle diagrams}.   These are pieces of link diagrams in the rectangle $\mathcal R$,  with some number of endpoints of curves on the top and bottom boundaries of $\mathcal R$, and a distinguished vertical segment representing the hole in the annulus cross interval. (We call this curve the flagpole.)  Affine tangle diagrams are regarded as equivalent if they are regularly isotopic;  see ~\cite{GH2} for details.  An affine $(n, n)$--tangle diagram is one with $n$ vertices (endpoints of curves) on the top, and $n$ vertices on the bottom edge of $\mathcal R$.  See Figure \ref{figure-affine tangle diagram}.  We label the vertices on the top edge from left to right as $\p 1,\dots \p n$ and those on the bottom edge from left to right as $\pbar 1, \dots, \pbar n$.  We order the vertices
by 
$
\p 1 < \p 2 < \cdots < \p n < \pbar n  < \cdots < \pbar 2 < \pbar 1.
$

\begin{definition}\rm\label{definition affine BMW}
 The {\em affine
Birman--Wenzl--Murakami} algebra
$\abmw  {n, S}$ is the
$S$ algebra with generators $y_1^{\pm 1}$, $g_i^{\pm 1}$  and
$e_i$ ($1 \le i \le n-1$) and relations:
\begin{enumerate}
\item (Inverses)\hods $g_i g_i\inv = g_i\inv g_i = 1$ and 
$y_1 y_1\inv = y_1\inv y_1= 1$.
\item (Idempotent relation)\hods $e_i^2 = \delta_0 e_i$.
\item (Type $B$ braid relations) 
\begin{enumerate}
\item[\rm(a)] $g_i g_{i+1} g_i = g_{i+1} g_ig_{i+1}$ and 
$g_i g_j = g_j g_i$ if $|i-j|  \ge 2$.
\item[\rm(b)] $y_1 g_1 y_1 g_1 = g_1 y_1 g_1 y_1$ and $y_1 g_j =
g_j y_1 $ if $j \ge 2$.
\end{enumerate}
\item[\rm(4)] (Commutation relations) 
\begin{enumerate}
\item[\rm(a)] $g_i e_j = e_j g_i$  and
$e_i e_j = e_j e_i$  if $|i-
j|
\ge 2$. 
\item[\rm(b)] $y_1 e_j = e_j y_1$ if $j \ge 2$.
\end{enumerate}
\item[\rm(5)] (Affine tangle relations)\vadjust{\vskip-2pt\vskip0pt}
\begin{enumerate}
\item[\rm(a)] $e_i e_{i\pm 1} e_i = e_i$,
\item[\rm(b)] $g_i g_{i\pm 1} e_i = e_{i\pm 1} e_i$ and
$ e_i  g_{i\pm 1} g_i=   e_ie_{i\pm 1}$.
\item[\rm(c)\hskip1.2pt] For $j \ge 1$, $e_1 y_1^{ j} e_1 = \delta_j e_1$. 
\vadjust{\vskip-
2pt\vskip0pt}
\end{enumerate}
\item[\rm(6)] (Kauffman skein relation)\hods  $g_i - g_i\inv = (q\inv-q)(e_i -1)$.
\item[\rm(7)] (Untwisting relations)\hods $g_i e_i = e_i g_i = \rho \inv e_i$
 and $e_i g_{i \pm 1} e_i = \rho  e_i$.
\item[\rm(8)] (Unwrapping relation)\hods $e_1 y_1 g_1 y_1 = \rho e_1 = y_1 
g_1 y_1 e_1$.
\end{enumerate}
\end{definition}

Let $X_1$,  $G_i$,  $E_i$  denote the following affine tangle diagrams:

$$
X_1 = \inlinegraphic[scale= .7]{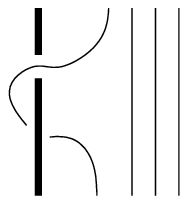}
\qquad
G_i =  \inlinegraphic[scale=.60]{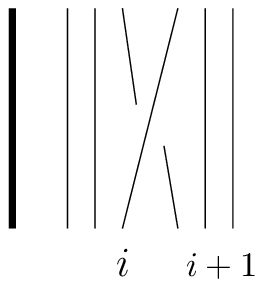}\qquad
E_i =  \inlinegraphic[scale= .7]{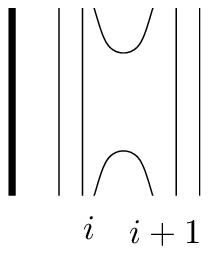} 
$$

\begin{theorem} [\cite{GH1}]
\label{theorem:  isomorphism affine BMW and KT}
The affine BMW algebra $\abmw {n, S}$ is isomorphic to the affine Kauffman tangle algebra $\akt {n, S} $ by a map $\varphi$ determined by
$\varphi(g_i) = G_i$,  $\varphi(e_i) = E_i$,   and $\varphi(y_1) =   \rho X_1$.
\end{theorem}

We now suppose $S$  (as above)  has additional distinguished invertible elements $u_1, \dots, u_r$.

 \begin{definition}
The {\em cyclotomic BMW algebra}  $\bmw{n, S, r}(u_1, \dots, u_r)$
is the quotient of $\abmw{n, S}$ by the relation
\begin{equation} \label{equation: cyclotomic relation1}
(y_1 - u_1)(y_1 - u_2) \cdots (y_1 - u_r) = 0.
\end{equation}
\end{definition}

To define the cyclotomic Kauffman tangle algebra, we begin by rewriting the relation Equation (\ref{equation: cyclotomic relation1})  in the form
$
\sum_{k = 0}^r  (-1)^{r-k}  \varepsilon_{r-k}(u_1, \dots, u_r) y_1^k = 0,
$
where $\varepsilon_j$  is the $j$--th elementary symmetric function.  The corresponding relation in the affine Kauffman tangle algebra is
$
\sum_{k = 0}^r  (-1)^{r-k}  \varepsilon_{r-k}(u_1, \dots, u_r) \rho^k X_1^k = 0,
$
Now we want to impose this as a local skein relation.

 \begin{definition}
The {\em cyclotomic Kauffman tangle algebra} $\kt{n, S, r}(u_1,  \dots, u_r)$ is the quotient of the affine Kauffman tangle algebra $\akt{n, S}$ by the
cyclotomic skein relation:
\begin{equation} \label{equation: kt cyclotomic relation}
\sum_{k = 0}^r  (-1)^{r-k}  \varepsilon_{r-k}(u_1, \dots, u_r) \rho^k \inlinegraphic[scale=.7]{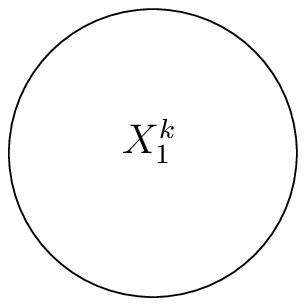} = 0,
\end{equation}
The sum is over affine tangle diagrams which differ only in the interior of  the indicated disc and
are identical outside of the disc;    the interior of the disc contains an interval on the flagpole and a piece of an affine tangle diagram isotopic to $X_1^k$.
\end{definition}

\begin{definition} Say that $S$ is {\em weakly admissible} if $e_1$ is not a torsion element in $\bmw {2, S, r}$.  Say that $S$ is {\em admissible}  if  $\{e_1, y_1 e_1, \dots, y_1^{r-1} e_1\}$ is linearly independent over $S$ in $\bmw {2, S, r}$.
\end{definition}

These conditions can be translated into explicit conditions on the parameters of $S$; see
~\cite{Wilcox-Yu, GH2, GH3}.

\begin{theorem}[\cite{GH2, GH3, Yu-thesis, Wilcox-Yu2}]  \label{theorem: cyclotomic isomorphism}
If $S$ is an admissible integral domain, then the assignment $e_i \mapsto E_i$, $g_i \mapsto G_i$,  $y_1 \mapsto \rho X_i$  determines an isomorphism of $\bmw{n, S, r}$ and
$\kt{n, S, r}$. Moreover these algebras are free $S$--modules of rank $r^n(2n-1)!!$.
\end{theorem}

Because of  Theorems \ref{theorem:  isomorphism affine BMW and KT} and
\ref{theorem: cyclotomic isomorphism},   we will no longer take care to distinguish between
 affine or cyclotomic BMW algebras and their realizations as algebras of tangles.  
We identify $e_i$ and $g_i$ with the corresponding affine tangle diagrams and $x_1 = \rho\inv y_1$ with the affine tangle diagram $X_1$.  The ordinary BMW algebra  $\bmw{n,S}$ imbeds in the affine BMW algebra $\bmw{n, S}$ as the subalgebra generated by the $e_i$'s and $g_i$'s.

\subsection{The rank of  tangle diagrams}   An ordinary or affine tangle diagram $T$ with $n$ strands is said to have {\em rank}   $\le r$ if it can be written as a product
$T = T_1 T_2$, where $T_1$ is an (ordinary or affine)  $(r, n)$  tangle and $T_2$  is
an (ordinary or affine)  $(n, r)$ tangle.

\subsection{The algebra involution $^*$ on BMW algebras}  Each of the ordinary, affine, and cyclotomic BMW algebras admits a unique involutive algebra anti--\break automorphism,  denoted  $a \mapsto a^*$,  fixing each of the generators $g_i$, $e_i$  (and $x_1$ in the affine or cyclotomic case).      For an (ordinary or affine)  tangle diagram $T$ representing an element of one of these algebras,  $T^*$  is the diagram obtained by flipping $T$ around a horizontal axis.

\subsection{The Hecke algebra and the BMW algebra}
\def\braid{\mathcal B}
\def\s{\sigma}
The  Hecke algebra $\hec {n,S}(q^2)$ of type $A$ is the quotient of the group algebra  
$S\, \braid_n$
of the braid group,  by the relations  $\s_i - \s_i\inv =  (q - q\inv) \quad  (1 \le i \le n-1),$
where $\s_i$ are the Artin braid generators.  Let $\tau_i$  denote the image of the braid generator
$\sigma_i$ in the Hecke algebra.

 Given an permutation $\pi \in \S_n$,   let $\beta_\pi$ be the
{\em positive permutation braid}  in the braid group $\braid_n$ whose image in $\S_n$ is $\pi$.
A positive permutation braid is a braid in which two strands cross an most once, and all crossings are positive, that is the braid is in the monoid generated by the Artin generators
$\sigma_i$ of the braid group.  Let $g_\pi$  be the image of $\beta_\pi$ in $\bmw  {n,S}$, and
$\tau_\pi$  the image of  $\beta_\pi$ in $\hec {n, S}(q^2)$.
If $\pi$ has a reduced expression $\pi = s_{i_1} s_{i_2} \cdots s_{i_\ell}$,  then
$g_\pi =  g_{i_1} g_{i_2} \cdots g_{i_\ell}$, and  $\tau_\pi =  \tau_{i_1} \tau_{i_2} \cdots \tau_{i_\ell}$.
It is well known that $\{\tau_\pi :  \pi \in \S_n\}$  is a basis of the Hecke algebra $\hec {n, S}(q^2)$.
The Hecke algebra has an involutive algebra anti-automorphism $x \mapsto x^*$  determined by $(\tau_\pi)^* = \tau_{\pi\inv}$.

\subsection{Affine and cyclotomic Hecke algebras}  \label{subsection: preliminaries, affine and cyclotomic Hecke algebras}

\begin{definition} (See ~\cite{ariki-book}.)
Let $S$ be a commutative unital ring with an invertible element $q$.  The {\em affine Hecke algebra} 
$\ahec{n,S}(q^2)$ 
over $S$
is the $S$--algebra with generators $t_1,  \tau_1,  \dots, \tau_{n-1}$, with relations:
\begin{enumerate}
\item  The generators $\tau_i$ are invertible,  satisfy the braid relations,  and \break $\tau_i - \tau_i\inv =  (q - q\inv)$.
\item  The generator $t_1$ is invertible,  $t_1 \tau_1 t_1 \tau_1 = \tau_1 t_1 \tau_1  t_1$  and $t_1$ commutes with $\tau_j$  for $j \ge 2$.
\end{enumerate}
Let $u_1,  \dots, u_r$  be additional invertible elements in $S$.   The {\em cyclotomic Hecke algebra}
$\hec{n, S, r}(q^2; u_1,  \dots, u_r)$  is the quotient of the affine Hecke algebra $\ahec{n, S}(q^2)$ by the polynomial relation    $(t_1 - u_1) \cdots (t_1 - u_r) = 0$.
\end{definition}

Define elements $t_j$  ($1 \le j \le n$)   in the affine or cyclotomic Hecke algebra by
$$
t_j =  \tau_{j-1} \cdots \tau_1 t_1 \tau_1 \cdots \tau_{j-1}.
$$
It is well known that the ordinary Hecke algebra $\hec {n, S}(q^2)$ imbeds in the affine  Hecke algebra and that the affine Hecke algebra $\ahec{n,S}(q^2)$  is a free $S$--module with basis the set of elements 
$\tau_\pi   \boldt^b$,  where $\pi \in S_n$  and $\boldt^b$  denotes a Laurent monomial in $t_1, \dots, t_n$.  Similarly, a  cyclotomic Hecke algebra $\hec{n, S, r}(q; u_1,  \dots, u_r)$ is a free $S$--module with basis the set of elements $\tau_\pi   \boldt^b$,  where now $\boldt^b$ is a monomial with restricted exponents 
$0 \le b_i \le r-1$.

Let $S$ be a commutative ring with appropriate parameters $\rho$, $q$, $\delta_j$.
There is an algebra homomorphism $p : \abmw {n, S} \rightarrow \ahec {n, S}(q^2)$ determined by $g_i \mapsto \tau_i$,  $e_i \mapsto 0$, and $x_1 \mapsto t_1$.
The kernel of $p$ is the ideal $I_n$ spanned by affine tangle diagrams with rank strictly less than $n$.  Suppose that $S$ has  additional parameters $u_1, \dots, u_r$.  Then $p$  induces a homomorphism of the cyclotomic quotients $p: \bmw {n, S, r}(u_1, \dots u_r) \rightarrow \hec {n, S, r}(q^2; u_1, \dots, u_r)$.

The affine and cyclotomic Hecke algebras have unique involutive algebra anti-automorphisms $^*$ fixing the generators $\tau_i$ and $t_1$.  (The image of a word in the generators is the reversed word.)  The quotient map $p$  respects the involutions, $p(x^*) = p(x)^*$.

We have a linear section $t:  \ahec {n, S}(q^2) \rightarrow \abmw {n, S}$
of the map $p$  determined by $t(\tau_\pi \boldt^b) =  g_\pi \boldx^b$.  Moreover, 
$t(x^*)  \equiv  t(x)^* \mod I_n$  and $t(x)t(y)  \equiv t(xy) \mod I_n$  for any $x, y \in \ahec {n, S}(q^2)$.
Analogous statements hold for the cyclotomic algebras.

\subsection{Cellular bases}
We recall the definition of {\em cellularity}  from ~\cite{Graham-Lehrer-cellular}; see also
~\cite{Mathas-book}.   The version of the definition given here is slightly weaker than the original definition in ~\cite{Graham-Lehrer-cellular}; we justify this below.

\begin{definition}  Let $R$ be an integral domain and $A$ a unital $R$--algebra.  A {\em cell datum} for $A$ consists of  an $R$--linear algebra involution $*$ of $A$; a partially ordered set $(\Lambda, \ge)$ and 
for each $\la \in \Lambda$  a set $\mathcal T(\lambda)$;  and   a subset $
\mathcal C = \{ c_{s, t}^\la :  \la \in \Lambda \text{ and }  s, t \in \mathcal T(\la)\} \subseteq A$; 
with the following properties:
\begin{enumerate}
\item  $\mathcal C$ is an $R$--basis of $A$.
\item  For each $\la \in \Lambda$,  let $\breve A^\la$  be the span of the  $c_{s, t}^\mu$  with
$\mu > \la$.   Given $\la \in \Lambda$,  $s \in \mathcal T(\la)$, and $a \in A$,   there exist coefficients 
$r_v^s( a) \in R$ such that for all $t \in \mathcal T(\la)$:
$$
a c_{s, t}^\la  \equiv \sum_v r_v^s(a)  c_{v, t}^\la  \mod  \breve A^\la.
$$
\item  $(c_{s, t}^\la)^* \equiv c_{t, s}^\la   \mod  \breve A^\la$ for all $\la\in \Lambda$ and, $s, t \in \mathcal T(\lambda)$.

\end{enumerate}
$A$ is said to be a {\em cellular algebra} if it has a  cell datum.  
\end{definition}

For brevity, we say write  that   $\mathcal C$ is a cellular basis of $A$.  

\begin{remark} \mbox{} \label{remark:  on definition of cellularity}
\begin{enumerate}
\item  The original definition in  ~\cite{Graham-Lehrer-cellular} requires that $(c_{s, t}^\la)^* = c_{t, s}^\la $ for all $\la, s, t$.  However, one can check that the basic consequences of the definition (\cite{Graham-Lehrer-cellular}, pages 7-13) remain valid with our weaker axiom.
\item  In case $2 \in R$ is invertible, one can check that our definition is equivalent to the original.
\item  One can formulate a version of the  \lq\lq basis--free" definition of cellularity of K\"onig and Xi  (see for example  ~\cite{KX-inflations})  equivalent to our modified definition.
\item  Suppose $A$ is an $R$--algebra with involution $*$, and $J$ is a $*$--closed ideal; then we have an induced algebra involution $*$ on $A/J$.  Let us say that $J$ is a cellular ideal in $A$ if it satisfies the axioms of a  cellular algebra (except for being unital)  with cellular basis 
$ \{ c_{s, t}^\la :  \la \in \Lambda_J \text{ and }  s, t \in \mathcal T(\la)\} \subseteq J$
and we have, as in point (2) of the definition of cellularity, 
$
a c_{s, t}^\la  \equiv \sum_v r_v^s(a)  c_{v, t}^\la  \mod  \breve J^\la
$
not only for $a \in J$ but also for $a \in A$.
If $J$ is a cellular ideal in $A$, and 
$A/J$ is  cellular  (with respect to the given involutions), then $A$ is cellular.   With the original definition of ~\cite{Graham-Lehrer-cellular},  this statement would be true only  if
$J$ has a $*$--invariant $R$--module complement in $A$.
\item   Yu  ~\cite{Yu-thesis} has also proved cellularity of the cyclotomic BMW algebras, using the original definition of cellularity of ~\cite{Graham-Lehrer-cellular};  at one point, her proof requires a more delicate analysis,  in order to obtain a $*$--invariant complement in 
$\bmw {n, S, r}$ of the kernel of $p: \bmw{n, S,r} \to  \hec{n, S, r}$.

\end{enumerate}
\end{remark}

\section{Some new bases of the affine and cyclotomic BMW algebras}
  The basis of cyclotomic BMW algebras that we produced in ~\cite{GH2} involved ordered monomials in the non--commuting but mutually conjugate elements
$$
x_j' = g_{j-1} \cdots g_1 x_1 g_1\inv \cdots g_{j-1}\inv.
$$
To obtain this basis, we first produced a basis of the {\em affine} BMW algebra consisting of affine tangle diagrams satisfying certain topological conditions.

Here we want to produce a new finite basis of the cyclotomic BMW algebras involving monomials in the commuting, but non--conjugate, elements
$$
x_j = g_{j-1} \cdots g_1 x_1 g_1 \cdots g_{j-1}.
$$
At an intermediate stage of the exposition, we will also use the elements
$$
x_j'' = g_{j-1}\inv \cdots g_1\inv  x_1 g_1 \cdots g_{j-1}, 
$$
see the following figure: 
$$
x_4 = \inlinegraphic{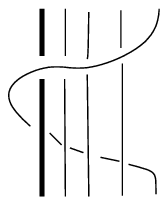}, \quad \quad x'_4 = \inlinegraphic{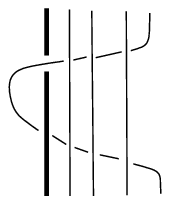}, \quad\quad x''_4 = \inlinegraphic{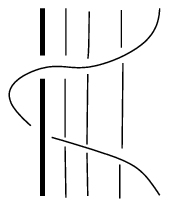}.
$$

\subsection{Flagpole descending affine tangle diagrams}
 
\begin{definition}\rm
\label{definition: orientation}
An {\em orientation} of an  affine 
  $(n,n)$--tangle diagram  is
  a linear ordering of the strands,
 a choice of an orientation of each strand, and a choice of an initial point on 
each closed loop.
   \end{definition}
 
 An orientation determines a way of traversing the tangle diagram;  namely,  
the 
strands are traversed successively, in
the given order and orientation (the closed loops being traversed starting at 
the assigned initial point).

  \begin{definition}\rm
  \label{definition: stratified}
  An oriented affine  $(n,n)$--tangle diagram is {\em stratified} if 
  \begin{enumerate}
  \item
  there is a linear ordering of the strands  such that if strand $s$ precedes strand $t$ in the order, then each crossing of $s$ with $t$ is an over--crossing. 
   \item  each strand is totally descending,  that is, each self--crossing of the strand is encountered first as an over--crossing as the strand is traversed according to the orientation.
   \end{enumerate}  
   We call the corresponding ordering of the strands the {\em stratification order}.  
  \end{definition}
  
  Note that a stratification order need not coincide with the ordering of strands determined by the orientation.   {\em In the rest of the paper, we are going to use the following orientation and stratification order  on affine tangle diagrams;  when we say an affine tangle diagram is oriented or stratified, we mean with respect to this orientation and stratification order.}

\begin{definition}  \label{definition:  verticals second orientation}  A {\em verticals--second orientation} of affine tangle diagrams is one in which:
\begin{enumerate}
\item Non-closed strands are oriented from lower to higher numbered vertex.
\item  Horizontal strands with vertices at the top of the diagram precede vertical strands,  and vertical strands precede horizontal strands with vertices at the bottom of the diagram.  Non-closed strands precede closed loops.
\item  Horizontal strands with vertices at the top of the diagram are ordered according to the order of their {\em final}  vertices.  Vertical strands 
 and horizontal strands with vertices at the bottom of the diagram are each ordered according to the order of their {\em initial} vertices.
\end{enumerate}
 A {\em verticals--second  stratification order} is one in which
the order of strands agrees with that of a verticals--second orientation, except that
vertical strands are ordered according to the  {\em reverse} order of their {\em initial} vertices.
\end{definition}

An affine tangle diagram without closed loops has a unique verticals--second orientation and a unique verticals--second stratification order.

A {\em simple winding} is a piece of an affine tangle
diagram with one ordinary strand, without self--crossings, 
 regularly isotopic to the intersection of one of the
affine tangle diagrams
$x_1$ or $x_1\inv$ with a neighborhood of the flagpole.

\begin{definition}\rm
An affine tangle diagram is in {\em standard
position} (See Figure \ref {figure-standard position}) if:
\begin{enumerate}
\item  It has no crossings to the left of the flagpole.
\item  There is a neighborhood of the flagpole whose intersection with
the tangle diagram is a union of simple windings.
\item  The simple windings have no crossings and are not nested.  That is,
between the two crossings of a simple winding with the flagpole, there is
no other crossing of a strand with the flagpole.
\end{enumerate}
\end{definition}

\begin{figure}[t]
$$\inlinegraphic{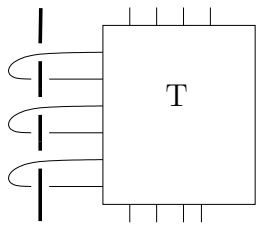}$$
\caption{Affine tangle diagram in standard position} \label{figure-standard position}
\end{figure}

  \begin{definition} An oriented, stratified affine tangle diagram $T$  in standard position is said to be {\em flagpole descending} if it satisfies the following conditions:
  \begin{enumerate}
   \item  $T$ is not regularly isotopic to an affine tangle diagram in standard position with fewer simple windings.
  \item  The strands of $T$ have no self--crossings.
  \item   As  $T$ is traversed according to the orientation,  successive
  crossings of ordinary strands with the flagpole descend the flagpole.   
  \end{enumerate}
  \end{definition}

\begin{proposition} \label{proposition:  flagpole descending verticals--second span}
 The affine BMW algebra $\abmw{n, S}$  is spanned by affine tangle diagrams without closed loops that are
flagpole descending and stratified.
\end{proposition}

 \begin{proof}  This follows  from ~\cite{GH2}, Proposition 2.19.
 \end{proof}

\subsection{$\Z$--Brauer diagrams and liftings in the affine BMW algebras}

We recall that a Brauer diagram is a tangle diagram in the plane,  in which information about over-- and under--crossings is ignored.  
  Let $G$ be a group.   A  {\em $G$--Brauer diagram}  (or {\em $G$--connector}) is
an Brauer diagram in which each strand is endowed with an orientation and labeled by an element of the group $G$.   Two labelings are regarded as the same if the orientation of a strand is reversed and the group element associated to the strand is inverted.

Define a map $c$ (the connector map) from oriented affine $(n,n)$--tangle diagrams  {\em without  closed loops}  to 
$\Z$--Brauer diagrams as follows.  Let $a$ be an oriented affine $(n,n)$--tangle diagram without closed loops.
 If $s$ connects two vertices
$\mathbold v_1$  to $\mathbold v_2$,   include a curve $c(s)$ in $c(a)$ connecting the same vertices with the same orientation, and label
 the oriented strand $c(s)$  with the {\em winding number} of $s$ with respect to the 
 flagpole.\footnote{The 
winding number $n(s)$ is determined 
 combinatorially as follows:  traversing the strand in its orientation, list the over--crossings $(+)$ and
under-crossings $(-)$ of the strand with the flagpole.  Cancel any
two successive  $+$'s or $-$'s in the list, so  the list now consists of
alternating $+$'s and $-$'s.  Then $n(s)$ is $\pm (1/2)$ the length of the
list, $+$ if the list begins with a $+$, and $-$ if the list begins with
a~$-$.}

\begin{lemma}[\cite{GH2}, Lemma  2.21]
\label{lemma:  isotopic flagpole descending diagrams}
Two  affine tangle diagrams without closed loops,  with the same $\Z$--Brauer diagram, both stratified and flagpole descending,  are  regularly isotopic.
\end{lemma}

The symmetric group $\S_n$  can be regarded as the subset of $(n,n)$--Brauer diagrams consisting of diagrams with only vertical strands.  $\S_n$
acts on ordinary or $\Z$-labelled $(n,n)$--Brauer diagrams on the left and on the right by the usual multiplication of diagrams,  that is,  by stacking diagrams.

We consider a particular family of permutations in $\S_n$.    Let $s$ be an integer,
$0 \le s \le n$,  with $s$ congruent to $n \mod 2$.  Write $f = (n-s)/2$.  Following Enyang
~\cite{Enyang2},  let 
\def\dfn{\mathcal D_{f,n}}
\def\enyang #1 #2{\mathcal D_{#1, #2}}
$\dfn$ be the set of permutations $\pi \in \S_n$ satisfying:
\begin{enumerate}
\item  If $i, j$  are even numbers with $2 \le i <  j \le 2f $,  then $\pi(i) < \pi(j)$.
\item  If  $i$ is odd with $1 \le i \le 2f-1$,  then $\pi(i) < \pi(i+1)$.
\item If $2f + 1 \le i < j \le n$,  then $\pi(i) < \pi(j)$.
\end{enumerate}
Then $\dfn$ is a complete set of left coset representatives of
$$((\Z_2 \times \cdots \times \Z_2) \rtimes \S_f)  \times \S_s  \subseteq \S_n,$$
where the $f$ copies of $\Z_2$  are generated by the transpositions $(2i-1, 2i)$ for
$1 \le i \le f$;   $\S_f$ permutes the $f$ blocks  $[2i-1, 2i]$ among themselves;  and
$\S_s$  acts on the last $s$ digits  $\{2f+1, \dots, n\}$.  

An element $\pi$ of $\dfn$ factors 
as $\pi =  \pi_1 \pi_2$,  where $\pi_2 \in  \enyang f f$, and $\pi_1$ is a $(2 f, s)$ shuffle; i.e., 
$\pi$ preserves the order of $\{1, 2, \dots, 2f\}$ and of $\{2f+1, \dots, 2f+s = n\}$.
Moreover,  $\ell(\pi) = \ell(\pi_1) + \ell(\pi_2)$.

For any $\Z$--Brauer diagram $D$, let $D_0$  denote the underlying ordinary Brauer diagram;  that is, $D_0$ is obtained from $D$ by forgetting the integer valued labels of the strands.   If $D$ is a $\Z$--Brauer diagram with exactly $s$ vertical strands, then
$D$ has a unique factorization 
\begin{equation} 
D =  \alpha \, d\ \, \beta\inv,  \label{equation: factorization of Brauer diagrams}
\end{equation}
where
 $\alpha$ and $\beta$ are elements of $\dfn$, and
 $d$  has underlying Brauer diagram of the form $d_0 = e_1 e_2\cdots e_{2f-1}  \pi,$  where $\pi$ is a permutation of $\{2f+1, \dots, n\}$.
This factorization is illustrated in Figure 
\ref{figure Brauer factorization}.

\begin{figure}[t]
$$\inlinegraphic[scale=.75]{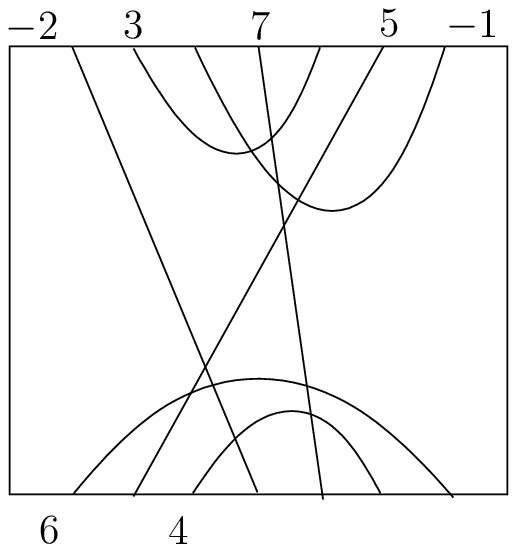} \quad = \inlinegraphic[scale=.75]{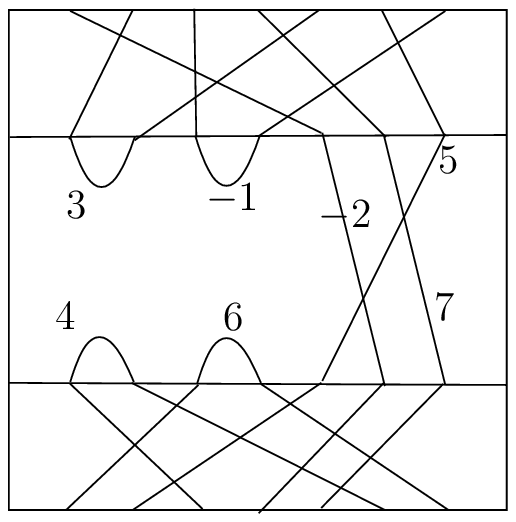}$$
\caption{Factorization of $\Z$--Brauer diagrams} \label{figure Brauer factorization}
\end{figure}

\def\Hom{{\rm Hom}}
\def\End{{\rm End}}
 It will be convenient to work in the affine BMW  category, that is,  the category whose objects are the natural numbers $0, 1, 2, \dots$,   with $\Hom(k, \ell)$ being the $S$--span of affine
 $(k, \ell)$--tangle diagrams, modulo Kauffman skein relations.  Let us introduce the elements  $\cup_i$ and $\cap_i$ which are the lower and upper half of $e_i$.

$$
\cup_i =  \inlinegraphic{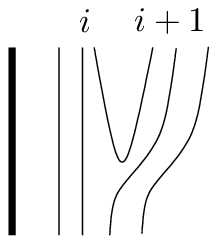}, \qquad
\cap_i = \inlinegraphic{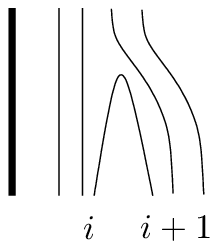}.
$$

We collect several elementary observations.  Fix integers $n$ and $s$  with
$0 \le s \le n$ and $s \equiv n \mod 2$.  Set $f = (n-s)/2$.   Each of the following statements is justified by picture proofs.

\begin{lemma} \label{lemma:  identities in the tangle category} \mbox{}
\begin{enumerate}
\item  $e_{1} e_3 \cdots  e_{2f-1} =( \cap_{2f-1} \cdots  \cap_3  \cap_1)  (\cup_1 \cup_3 \cdots \cup_{2f-1})$.

\ignore{
\item  For $2f+1 \le k \le n$,    $(\cup_1 \cup_3 \cdots \cup_{2f-1}) x'_{k} =
x'_{k-2f}  (\cup_1 \cup_3 \cdots \cup_{2f-1})$.
}

\item  For $k$ odd, $1 \le k \le 2f-1$,     $(\cup_1 \cup_3 \cdots \cup_{2f-1}) x'_{k} =
  (\cup_1 \cup_3 \cdots \cup_{2f-1}) x_{k} $.

\item  For $k$ odd, $1 \le k \le 2f-1$,     $ x''_{k} (  \cap_{2f-1} \cdots  \cap_3  \cap_1  ) =
 x_{k}  ( \cap_{2f-1} \cdots  \cap_3  \cap_1) $.

  \item   If $\pi$ is a a permutation of $\{2f+1, 2f+2, \dots, n\}$, then
   $$(\cup_1 \cup_3 \cdots \cup_{2f-1}) g_\pi=
 g_{\tilde \pi} (\cup_1 \cup_3 \cdots \cup_{2f-1}),$$
  where  $\tilde \pi$ is the permutation of $\{1, 2, \dots, s\}$ defined by 
  $\tilde \pi(j)  = \pi(j + 2f) - 2f$.  More generally,  if $T$ is an ordinary tangle on the
  strands $\{2f+1, 2f+2, \dots, n\}$, then
    $$(\cup_1 \cup_3 \cdots \cup_{2f-1})\  T=
\tilde T \ (\cup_1 \cup_3 \cdots \cup_{2f-1}),$$
where $\tilde T$ is the shift of $T$ to the strands $\{1, 2, \dots, s\}$.

\item  For $1 \le k \le s$,    $(\cup_1 \cup_3 \cdots \cup_{2f-1}) x_{k+ 2 f} =
x_{k}  (\cup_1 \cup_3 \cdots \cup_{2f-1})$.

\end{enumerate}
\end{lemma}

Now we can obtain a lifting of $\Z$--Brauer diagrams to affine tangle diagrams that are flagpole descending  and stratified,  using the factorization
of Equation (\ref{equation: factorization of Brauer diagrams}).  
Let $D$ be a $\Z$--Brauer diagram with exactly $s$ vertical strands. Set $f = (n-s)/2$.
  Consider the  factorization $D = \alpha \, d  \, \beta\inv,$  where
$\alpha, \beta \in \dfn$, and $d_0 = e_1 \cdots e_{2f-1} \pi$,  with $\pi$ a permutation
of $\{2f+1, \dots, n\}$.

 First,  there is a unique (up to regular isotopy) stratified ordinary $(n, n)$--tangle diagram $T_{d_0}$  without closed loops or self--crossings of strands with Brauer diagram
$$
c(T_{d_0}) = d_0 =  e_1 e_3 \cdots e_{2f-1}  \pi,
$$
namely 
\begin{equation}
\begin{split}
T_{d_0} &= e_1 e_3 \cdots e_{2f-1}  g_\pi  = ( \cap_{2f-1} \cdots  \cap_3  \cap_1) (\cup_1 \cup_3 \cdots \cup_{2f-1}) g_\pi \\
&=  ( \cap_{2f-1} \cdots  \cap_3  \cap_1) g_{\tilde \pi}  (\cup_1 \cup_3 \cdots \cup_{2f-1}),\\
\end{split}
\end{equation}
  where  $\tilde \pi$ is the permutation of $\{1, 2, \dots, s\}$ defined by 
  $\tilde \pi(j)  = \pi(j + 2f) - 2f$. 

Next, we set  
\begin{equation} \label{equation: definition of Td}
\begin{split}
T'_{d} &=  (x''_1)^{a_1} \cdots (x''_{2f-1})^{a_{2f-1}} ( \cap_{2f-1} \cdots  \cap_3  \cap_1)  g_{\tilde \pi}   (x''_{s})^{b_{s}} \cdots     (x''_1)^{b_1} \\
&\phantom{ = (x_n')^{a_n}xxxx}   (\cup_1 \cup_3 \cdots \cup_{2f-1})  ( x_{2f-1}')^{c_{2f-1}} \cdots  ( x_1')^{c_1},
\end{split}
\end{equation}
where the exponents are determined as follows:   
\begin{enumerate}
\item For $i$ odd, $i \le 2f-1$,  if $d$  has a strand beginning
at $\p i$  with label $k$,  then $c_i = k$;  otherwise $c_i = 0$.   
\item For $i \ge 2f+1$,  if  $d$  has a strand beginning
at $\p i$  with label $k$,  then $b_{i-2f} = k$;  otherwise $b_{i-2f} = 0$.  
\item For $i$ odd, $i \le 2f-1$,  if $d$  has a horizontal strand ending at 
$\pbar i$ with label $k$,  then  $a_i = k$;  otherwise,  $a_i = 0$.
\end{enumerate}

Finally, we set 
\begin{equation} \label{equation: definition of UprimeD}
T'_D =  g_\alpha \ T'_d  \  (g_\beta)^*,
\end{equation}
endowed with the verticals--second orientation.

\begin{example}  For the $\Z$--Brauer diagram $D$ illustrated in Figure \ref{figure Brauer factorization},  $T'_D$ is illustrated in  Figure \ref{figure: lifting 1}  (where the winding number of the strands are indicated by the integers written at the left of the figure.)
We have
\begin{equation*}
\begin{split}
T'_{d}  &=    (x''_1)^4  (x''_3)^6(\cap_1 \cap_3) g_1 g_2  (x''_3)^{5} (x''_2)^{7}  (x''_1)^{-2}   (\cup_1 \cup_3) 
 (x'_3)^{-1} (x'_1)^3 \\
\end{split}
\end{equation*}
\end{example}

\begin{figure}[ht]
\centerline{$\inlinegraphic{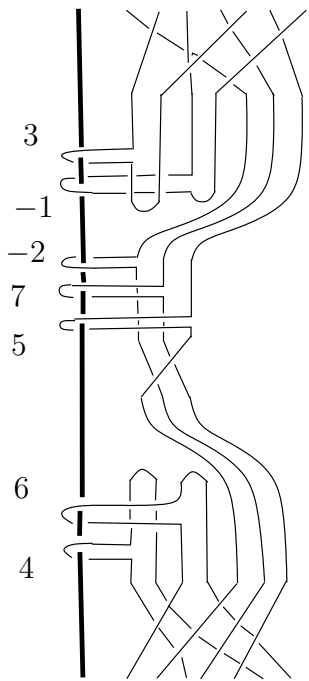}$}
\caption{Lifting of an affine Brauer diagram} \label{figure: lifting 1}
\end{figure}

\begin{lemma}  \label{lemma:  UprimeD properties}

$T'_D$ is     flagpole descending and stratified,  and has $\Z$--Brauer diagram equal to $D$.
\end{lemma}

\begin{proof}  Straightforward.
\end{proof}

\begin{proposition} \label{proposition:  U-prime spanning set of affine algebra}
$\U' =  \{ T'_D : D \text{ is a }  \Z  \text {--Brauer diagram}\}$ spans  $\abmw {n, S}$.
\end{proposition}

\begin{proof}  Follows from Proposition   \ref{proposition:  flagpole descending verticals--second span}, Lemma    \ref{lemma:  isotopic flagpole descending diagrams},  and Lemma    \ref{lemma:  UprimeD properties}.
\end{proof}

\begin{proposition}  \label{theorem: U-prime basis of affine algebra}
$\U' = \{ T'_D : D \text{ is a }  \Z  \text {--Brauer diagram}\}$ is a basis of $\abmw {n, S}$.
\end{proposition}

\begin{proof}  Essentially the same as the proof of Theorem 2.25
in ~\cite{GH2}.
\end{proof}

Now fix an integer $r \ge 1$.   Let $\U'_r$ be the set of $U'_D \in \U'$ such that the
integer valued labels on the strands of $D$ are restricted to lie in the interval
$0 \le k \le r-1$.  Equivalently, the exponents of the $x'_j$, $x''_j$ appearing in $U'_D$ are restricted to be in the same interval of integers.

\begin{proposition} \label{proposition: U-prime spanning set of cyclotomic algebra}
 The cyclotomic BMW algebra
$\bmw{n, S, r}(u_1,  \dots, u_r)$  is spanned over $S$ by $\U'_r$.
\end{proposition}

\begin{proof}  Same as the proof of Proposition 3.6
in ~\cite{GH2}.
\end{proof}

\begin{proposition} \label{theorem: U-prime basis of cyclotomic algebra} 
 For any  integral domain $S$ with admissible parameters, $\U'_r$ is an $S$--basis of the
 cyclotomic BMW algebra
$\bmw{n, S, r}(u_1,  \dots, u_r)$.
\end{proposition}

\begin{proof}  Same as the proof of Theorem 5.5
in ~\cite{GH2}.
\end{proof}

\begin{remark}  \label{remark:  stratified flagpole descending k-ell diagrams}

It is straighforward to generalize the content of this section to
affine $(k, \ell)$--tangle diagrams.  The notions of standard position, orientation, and stratification, and in particular the verticals--second orientation and stratification extend to
affine $(k, \ell)$--tangle diagrams.  Likewise, the notion of flagpole descending extends.  \ignore{Let $\abmw{}(k, \ell)$  denote the $S$--module of 
affine $(k, \ell)$--tangle diagrams,  modulo Kauffman skein relations and closed loop relations.  Then the analogue of Proposition \ref{proposition:  flagpole descending verticals--second span}  holds for $\abmw{}(k, \ell)$. }  

Define $(k, \ell)$--connectors to be ``Brauer diagrams"  with $k$ upper vertices and $\ell$ lower vertices, and likewise define $\Z$--weighted $(k, \ell)$--connectors. We can extend the definition of the connector map $c$  to a map from oriented affine $(k, \ell)$--tangle diagrams without closed loops to 
$\Z$--weighted $(k, \ell)$--connectors.  Then the analogue of Lemma \ref{lemma:  isotopic flagpole descending diagrams} holds.

Other results in this section can also be generalized, but we 
 we will  need only a weak version   of Propositon \ref{proposition:  U-prime spanning set of affine algebra},  and only for affine $(0, 2 f)$--tangle diagrams.
Consider the set of affine $(0, 2 f)$--tangle diagrams of the form
\begin{equation} \label{equation: canonical form for 0-2f tangles}
g_\alpha (x''_1)^{a_1} \cdots (x''_{2f-1})^{a_{2f-1}} ( \cap_{2f-1} \cdots  \cap_3  \cap_1),
\end{equation}
where $\alpha \in \enyang f f$.   These affine tangle diagrams are stratified and flagpole descending, and have no closed loops.  Moreover,  every $\Z$--weighted $(0, 2 f)$--connector has a lifting in this set.  Therefore,  by the analogue of Lemma \ref{lemma:  isotopic flagpole descending diagrams},  every totally descending, flagpole descending affine $(0, 2f)$--tangle diagram without closed loops, is regularly isotopic to one of the  diagrams represented in Equation (\ref{equation: canonical form for 0-2f tangles}).

\end{remark}

\subsection{New bases}  \label{section:  new bases with x-i}

So far, we have produced
 bases $\U'$ of the affine BMW algebras and $\U'_r$ of the cyclotomic BMW algebras involving  (as did our previous bases in ~\cite{GH2}) ordered monomials in the non--commuting  but conjugate elements $x'_j$ and $x''_j$.  
 
 We will now use these bases  to obtain new bases    involving instead monomials in the commuting elements $x_j$.
 
Consider the definition of $T'_D$  in Equations (\ref{equation: definition of Td}) and (\ref{equation: definition of UprimeD}).
Note that
\begin{equation*}
\begin{split}
(\cup_1 &\cup_3 \cdots \cup_{2f-1})  (x'_{2f-1})^{c_{2f-1}} \cdots (x'_3)^{c_3}  (x'_1)^{c_1} \\
&= (\cup_1 \cup_3 \cdots \cup_{2f-1})  (x_{2f-1})^{c_{2f-1}} \cdots (x'_3)^{c_3}  (x'_1)^{c_1}  \\
&= (\cup_1 \cup_3 \cdots \cup_{2f-1})  ((x'_{2f-3})^{c_{2f-3}} \cdots (x'_3)^{c_3}  (x'_1)^{c_1})   (x_{2f-1})^{c_{2f-1}},
\end{split}
\end{equation*}
using Lemma \ref{lemma:  identities in the tangle category} (2),
 and the commutivity of $x_{2f -1}$ with $\abmw {2f-2, S}$.  
Applying this step repeatedly, we end with
\begin{equation*}
\begin{split}
(\cup_1 &\cup_3 \cdots \cup_{2f-1})  (x'_{2f-1})^{c_{2f-1}} \cdots (x'_3)^{c_3}  (x'_1)^{c_1} \\
&= (\cup_1 \cup_3 \cdots \cup_{2f-1})  (x_1)^{c_1} (x_3)^{c_3}   \cdots    (x_{2f-1})^{c_{2f-1}} \\
\end{split}
\end{equation*}
Likewise,  using   Lemma \ref{lemma:  identities in the tangle category} (3),
\begin{equation*}
\begin{split}
 (x''_1)^{a_1} &\cdots (x''_{2f-1})^{a_{2f-1}}( \cap_{2f-1} \cdots  \cap_3  \cap_1) \\
& =  (x_1)^{a_1} \cdots (x_{2f-1})^{a_{2f-1}} ( \cap_{2f-1} \cdots  \cap_3  \cap_1)
\end{split}
\end{equation*}

Thus the  element $T'_D$  has the form
\begin{equation} \label{equation:  special form of UprimeD, 1}
\begin{split}
T'_D = g_\alpha & \   (x_1)^{a_1} \cdots (x_{2f-1})^{a_{2f-1}} ( \cap_{2f-1} \cdots  \cap_3  \cap_1)\ T  \\  
&(\cup_1 \cup_3 \cdots \cup_{2f-1}) \   (x_1)^{c_1} (x_3)^{c_3}   \cdots    (x_{2f-1})^{c_{2f-1}} \ (g_\beta)^*, \\
\end{split}
\end{equation}
where  $T$ is an affine tangle diagram
on $s$ strands with no horizontal strands.

We consider the cyclotomic BMW algebra  $W_n = \bmw{n, S, r}(u_1,  \dots, u_r)$;  the arguments for the affine BMW algebras are similar.  If $T'_D \in \U'_r$, then the exponents $a_i$ and $c_i$  of $x_i$  
in Equation (\ref{equation:  special form of UprimeD, 1}) satisfy
$0 \le a_i , c_i \le r-1$.      

Recall from Section \ref{subsection: preliminaries, affine and cyclotomic Hecke algebras}
that 
the quotient of $W_s$ by the ideal $I_s$ spanned by affine tangle diagrams with rank strictly less than $s$ is isomorphic to the cyclotomic Hecke algebra  
$\hec{s, S, r} = \hec{s, S, r}(q^2; u_1, \dots, u_r)$.  The cyclotomic Hecke algebra $\hec{s, S, r} $  is a free $S$--module with basis the set of 
$\tau_{ \omega}  t_1^{b_1} \cdots t_s^{b_s}$,  with $0 \le b_i \le r-1$ and $ \omega \in \S_s$.  Therefore, the element
$T \in W_s$ is congruent modulo $I_s$ to a linear combination of elements 
$t(\tau_{ \omega}  t_1^{b_1} \cdots t_s^{b_s}) = g_{ \omega}  x_1^{b_1} \cdots x_s^{b_s}$. 

 If we replace $T$ with $g_{ \omega}  x_1^{b_1} \cdots x_s^{b_s}$ in Equation (\ref{equation:  special form of UprimeD, 1}), and then apply
Lemma \ref{lemma:  identities in the tangle category} (1), (4) and (5), we obtain (in place of $T'_D$) an expression
\begin{equation} \label{equation: special form of U_D}
\begin{split}
 g_\alpha  \   x_1^{a_1} &\cdots x_{2f-1}^{a_{2f-1}} (e_1 e_3 \cdots e_{2f-1}  g_\omega) \\
 &  x_1^{c_1} x_3^{c_3}   \cdots    x_{2f-1}^{c_{2f-1}} x_{2f+1}^{b_1} \cdots x_n^{b_s} \ (g_\beta)^*
 \end{split}
\end{equation}

On the other hand,  if we replace $T$ by an element of $I_s$,  then we obtain (in place of $T'_D$)  a linear combination of 
affine tangle diagrams with rank strictly less than $s$.

Given a   $\Z$--Brauer diagram $D$, we define an element $T_D$ of the form displayed in Equation
(\ref{equation: special form of U_D}) whose associated $\Z$--Brauer diagram $c(U_D)$ is equal to $D$, as follows:
Suppose $D$ has $2 n$ vertices and $s$ vertical strands, and let $f = \break (n-s)/2$.
Let $D$ have the   factorization $D = \alpha \, d  \, \beta\inv,$  where
$\alpha, \beta \in \dfn$, and $d_0 = e_1 \cdots e_{2f-1} \pi$,  with $\pi$ a permutation
of $\{2f+1, \dots, n\}$. 

Define
\begin{equation} \label{equation: definition of Td in affine case}
\begin{split}
T_d =  x_1^{a_1} \cdots x_{2f-1}^{a_{2f-1}} &(e_1 e_3 \cdots e_{2f-1}  g_\pi) \\
 &  x_1^{c_1} x_3^{c_3}   \cdots    x_{2f-1}^{c _{2f-1}} x_{2f+1}^{b_{2f+1}} \cdots x_n^{b_n},
\end{split}
\end{equation}
where the exponents are determined as follows:  If $d$ has a horizontal strand beginning at $\p i$ with integer valued label $k$, then $c_i = k$; and $c_i = 0$ otherwise.   If $d$ has a vertical strand beginning at  $\p i$ with integer valued label $k$, then $b_i = k$; and $b_i = 0$ otherwise.
If $d$ has a horizontal strand ending at
$\pbar i$ with integer valued label $k$,  then $a_i = k$; and $a_i = 0$ otherwise.

Finally, set 
$
T_D = g_\alpha \ T_d \ ( g_\beta)^*.
$
Then
$c(U_D) = D$.

\begin{theorem} \label{theorem:  the x-i basis of cyclotomic BMW}  Let $S$ be an admissible integral domain.
 Let $\U_r$ be the set of $T_D$  corresponding to $\Z$--Brauer diagrams $D$ with integer valued labels in the interval  $0 \le k \le r-1$.   Then $\U_r$ is an $S$--basis of $\abmw {n, S, r}(u_1, \dots, u_r)$.
\end{theorem}

\begin{proof}
We will show that $\U_r$ is spanning.  Linear independence is proved as in the proof 
of Theorem 5.5
in ~\cite{GH2}.  It suffices to show that $\U'_r$ is contained in the linear span of $\U_r$.    

Let $T'_D \in \U'_r$, where $D$ has exactly $s$ vertical strands.  We show by induction on $s$ that $T'_D$ is in the span of $\U_r$.   If $s = 0$,   then in Equation (\ref{equation:  special form of UprimeD, 1}),   the tangle $T$ is missing, and  $T'_D$ is already an element of $\U_r$.  If $s = 1$,  then in 
 Equation (\ref{equation:  special form of UprimeD, 1}),   the tangle $T$ is equal to a power of  $x_1$,
 and again  $T'_D \in \U_r$,  by Lemma \ref{lemma:  identities in the tangle category} (5).
 
Assume that $s > 1$ and that all that all elements of $\U'_r$ with fewer than $s$ vertical strands are in the span of $\U_r$.    It follows from Equation (\ref{equation:  special form of UprimeD, 1}), and the discussion following it,  that $T'_D$ is in the span of $\U_r$,   modulo the ideal 
$I_n^{(s-1)}$ in $\abmw {n, S, r}(u_1, \dots, u_r)$ spanned by
of affine tangle diagrams with rank strictly less than $s$.

Now it only remains to check that $I_n^{(s-1)}$ is spanned by elements of $\U'_r$ with fewer than $s$ vertical strands.  Here one only has to observe that smoothing any crossing in a tangle diagram with 
$k$ vertical strands produces a tangle diagram with at most $k$ vertical strands.  Therefore, the
algorithm from ~\cite{GH2},  Propositions 2.18 and 2.19,
 for writing an affine tangle diagram (with $k$ vertical strands)  as a linear combination of flagpole descending affine tangle diagrams    produces only affine tangle diagrams with at most $k$ vertical strands.
\end{proof}

For the affine case, we have the following result, with essentially the same proof: 

\begin{theorem} 
Let $S$ be any ring with appropriate parameters.  $\U =  \{ T_D : D \text{ is a }  \Z  \text {--Brauer diagram}\}$ is a basis of   $\abmw {n, S}$.
\end{theorem}

Let $D$ be an  $\Z$--Brauer diagram, with $2n$ vertices and  $s$ vertical strands,  having factorization $D = \alpha d \beta$, and let 
$T_d$   be defined as in Equation
(\ref{equation: definition of Td in affine case}) and let  $T_D = g_\alpha\, T_d \,(g_\beta)^*$.   We can rewrite $T_D$ as follows.   Factor $\alpha$ as
$\alpha = \alpha_1 \alpha_2$, with $\alpha_1$ a $(2 f, s)$--shuffle, and $\alpha_2 \in \enyang  f f$, and factor $\beta$ similarly.   Then
\begin{equation} \label{equation:  rewriting T-D 1}
T_D =g_{\alpha_1}   \left [  g_{\alpha_2}  \xbold^a (e_1 e_3 \cdots e_{2f-1})  \xbold^c  (g_{\beta_2})^* \right ]   (g_\pi  \xbold^b)   (g_{\beta_1})^*,
\end{equation}
where $\xbold^a$ is short for $x_1^{a_1} \cdots x_{2f-1}^{a_{2f-1}}  $, and similarly for
$\xbold^c$,  while $\xbold^b$  denotes \break $x_{2f +1}^{b_{2f + 1}} \cdots x_n^{b_n}$.

\newpage
\section{Cellular bases of  cyclotomic BMW algebras}
\label{section:  cellular bases}

\subsection{Tensor products of affine tangle diagrams} 
\label{subsection: tensor products of affine tangle diagrams}
 The category of affine $(k, \ell)$--tangle diagrams is not a tensor category in any evident fashion.  Nevertheless, we can define a tensor product of affine tangle diagrams, as follows.  Let $T_1$ and $T_2$ be 
affine tangle diagrams  (say of size $(a, a)$  and $(b, b)$, respectively), and suppose that  $T_2$ has no closed loops.  Then
$T_1 \odot T_2$ is obtained by replacing the flagpole in the affine tangle diagram $T_2$ with the entire affine tangle diagram $T_1$.  See Figure \ref{figure: odot tensor product of affine tangle diagrams}.   If we regard $T_1$ and $T_2$ as representing framed
tangles in the annulus cross the interval $A \times I$,  then $T_1 \odot T_2$ is obtained by inserting the entire  copy of $A \times I$ containing $T_1$  into the hole of the copy of 
$A \times I$ containing $T_2$.

\begin{figure}[t] \label{figure: odot tensor product of affine tangle diagrams}
$$
\inlinegraphic[scale= 1.5]{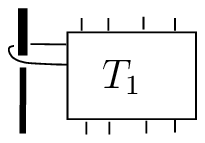},  \quad  \inlinegraphic[scale=1.5]{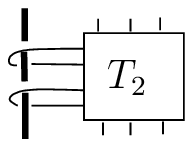} \mapsto$$
$$\inlinegraphic[scale= 1.5]{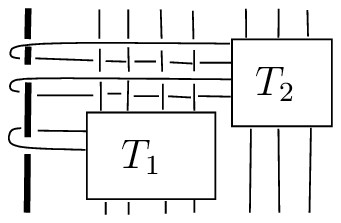} $$
\caption{$T_1 \odot T_2$}
\end{figure}

Then  $T_1 \otimes T_2 \mapsto T_1 \odot T_2$  determines a linear map from
$\abmw {a, S} \otimes \abmw {b, S}$ into $\abmw {a + b, S}$, or from
$\bmw{a, S,r}  \otimes \bmw{b, S, r}$ into $\bmw{a + b, S, r}$.
Note that $(T_1 \odot T_2)^* = T_1^* \odot T_2^*$.

These maps of affine and cyclotomic BMW algebras are not algebra homomorphisms.  In fact,
$$(1 \odot e_1)(1\odot x_1)(1\odot e_1) = z \odot e_1,$$  where $z$ is a (non--scalar) central element in $\abmw {a, S}$.    Nevertheless, we have
$$
(A \odot B)(S \odot T) = A S \odot B T,
$$
if no closed loops are produced in the product $BT$,  in particular, if  at least one of $B$ and $T$ has no horizontal strands.

\subsection{Cellular bases}

Using Equation (\ref{equation:  rewriting T-D 1}) and our remarks in Section
\ref{subsection: tensor products of affine tangle diagrams},  we can rewrite the elements $T_D$ of Section \ref{section:  new bases with x-i}  in the form
\begin{equation} \label{equation:  rewriting T-D 2}
T_D =g_{\alpha_1}   (\left [  g_{\alpha_2}  \xbold^a (e_1 e_3 \cdots e_{2f-1})  \xbold^c  (g_{\beta_2})^* \right ] \odot  g_\pi  \xbold^b  ) (g_{\beta_1})^*\\
\end{equation}
Here,  $D$ is a $\Z$--Brauer diagram with $s$ vertical strands and  $f = (n-s)/2$;  $\alpha_1$ and $\beta_1$  are  $(n-s, s)$--shuffles; $\pi \in \S_s$  and $\xbold^b = 
x_1^{b_1} \cdots x_s^{b_s}$.
 Moreover,  $\alpha_2$ and $\beta_2$ are  elements of $\enyang f f$,  
$\xbold^a = x_1^{a_1} x_3^{a_3} \cdots x_{2f -1}^{a_{2f -1}}$, and similarly for
$\xbold^c$.      

The affine $(2 f, 2f)$--tangle diagram
$$
T = g_{\alpha_2}  \xbold^a (e_1 e_3 \cdots e_{2f-1})  \xbold^c  (g_{\beta_2})^*
$$
is stratified and  flagpole descending, with no vertical strands and no closed loops.  Conversely,  any  stratified and  flagpole descending
affine $(2 f, 2f)$--tangle diagram with no vertical strands and no closed loops is regularly isotopic to one of this form.

Note that we can factor $T$ as $T = x y^*$,  where $x$ and $y$ are 
stratified and  flagpole descending
affine $(0, 2f)$--tangle diagram with no closed loops, namely  
$$x = g_{\alpha_2}  \xbold^a( \cap_{2f-1} \cdots  \cap_3  \cap_1)  \text{ and }
y = g_{\beta_2}  \xbold^c( \cap_{2f-1} \cdots  \cap_3  \cap_1).$$
By Remark \ref{remark:  stratified flagpole descending k-ell diagrams},  any stratified and  flagpole descending
affine $(0, 2f)$--tangle diagram with no closed loops is regularly isotopic to one of this form.

\begin{lemma}  \label{lemma: rewriting T-D 3}
The set of $T_D \in \U$  with  $s$ vertical strands equals the set of elements
$$
g_\alpha (x y^* \odot g_\pi \xbold^b) )   (g_\beta)^*,
$$
where  $x, y$  are stratified, flagpole descending affine $(0, n-s)$--tangle diagrams without closed loops or self-crossings of strands;  $\alpha$ and $\beta$ are
$(n-s, s)$--shuffles;  $\pi \in \S_s$, and $\xbold^b =   x_1^{b_1} \cdots x_s^{b_s}$.

Moreover,  $T_D \in \U_r$ if, and only if,   the exponents $b_i$ are in the range
$0 \le b_i \le r-1$,  and the winding numbers of $x$ and $y$ with the flagpole are in the same range.
\end{lemma}

We will show that the cyclotomic BMW algebras  defined over integral, admissible rings are cellular.
We fix an integral domain $S$ with admissible parameters, and write $\bmw {n, S, r}$  for
 $\bmw {n, S, r}(u_1, \dots, u_r)$ and $\hec {n, S, r}$ for 
  $\hec {n, S, r}(q^2; u_1, \dots, u_r)$.
  
  For each $s$ with $s \le n$ and $n-s$ even,  let $V_n^s$  be the span in 
  $\bmw {n, S, r}$ of the set of  elements $T_D \in \U_r$ with $s$ vertical strands.

\begin{lemma}  \label{lemma:  basis of cyclotomic Hecke to basis of cyclotomic BMW}
For each $s$, let $\mathbb B_s$ be a basis of $\hec {s, S, r}$.
Let $\Sigma_s$ be the set of elements
$$
g_\alpha (x y^* \odot t(b))    (g_\beta)^* \in \bmw {n, S, r},
$$
such that  $x, y$  are stratified, flagpole descending affine $(0, n-s)$--tangle diagrams without closed loops;  $\alpha$ and $\beta$ are
$(n-s, s)$--shuffles; and $b \in \mathbb B_s$.  Then $\Sigma_s$  is a basis of $V_n^s$.
\end{lemma}

\begin{proof}  Recall that $\{\tau_\pi \tbold^b : \pi \in \S_s \text{ and }   0 \le b_i \le r-1\}$
is a basis of $\hec {s, S, r}$,  and that $g_\pi \xbold^b = t(\tau_\pi \tbold^b)$.
 It follows from this and from Lemma \ref{lemma: rewriting T-D 3} that
$V_n^s$  is the direct sum over $(\alpha, \beta, x, y)$ of
$$
V_n^s(\alpha, \beta, x, y) =  \{ g_\alpha ( x y^* \odot t(u))  (g_\beta)^* : u \in \hec {s, S,r}\}
$$
 and that
$u \mapsto  g_\alpha ( T \odot t(u))  (g_\beta)^*$ is injective.  This implies the result.
\end{proof}

For each  $s$ ($s \le n$ and $n-s$ even),  let $(\mathcal C_s, \Lambda_s)$ be a cellular basis of the cyclotomic Hecke algebra $\hec {s, S,r}$.   Let $\Lambda = \{(s, \la) :  \la \in \Lambda_s\}$  with partial order
$(s, \la)  \ge (t, \mu)$  if $s < t$  or if $s = t$ and $\la \ge \mu$ in $\Lambda_s$.  For each pair
$(s, \la) \in \Lambda$,  we take $\mathcal T(s, \la)$ to be the set of triples
$(\alpha, x, u)$,  where $\alpha$ is a $(n-s, s)$--shuffle;   $x$ is a stratified, flagpole descending affine $(0, n-s)$--tangle without closed loops;  and $u \in \mathcal T(\lambda)$.  Define
$$
c_{(\alpha, x, u), (\beta, y, v)}^{(s, \la)} =  g_\alpha ( x y^* \odot t(c_{u, v}^\la))  (g_\beta)^*,
$$
and $\mathcal C$ to be the set of all 
$
c_{(\alpha, x, u), (\beta, y, v)}^{(s, \la)}
$.

\begin{lemma} \label{lemma: * property of cellular basis}
 $(c_{(\alpha, x, u), (\beta, y, v)}^{(s, \la)})^* \equiv c_{(\beta, y, v), (\alpha, x, u)}^{(s, \la)} \mod  \breve W_{n, S,r}^{(s, \la)}$.
\end{lemma}

\begin{proof}  $(g_\alpha ( x y^* \odot t(c_{u, v}^\la))  (g_\beta)^*)^* = g_\beta(y^* x \odot ( t(c_{u, v}^\la)^*) (g_\alpha)^*$, and $( t(c_{u, v}^\la)^* \equiv t(c_{v, u}^\la)$  modulo the span of diagrams of rank $< s$.  Hence $(c_{(\alpha, x, u), (\beta, y, v)}^{(s, \la)})^* \equiv c_{(\beta, y, v), (\alpha, x, u)}^{(s, \la)} $  modulo the span of diagrams of rank $<s$.
\end{proof}

\begin{lemma}  \label{lemma:  cellular cyclotomic BMW}
For any  affine $(n-s, n-s)$--tangle diagram $A$ and affine  $(s, s)$--tangle diagram $B$,
$(A \odot B)( x y^* \odot t(c_{u, v}^\la)) $
 can be written as a linear combination of elements
$
 ( x' y^* \odot t(c_{u', v}^\la)),
$
modulo $\breve W_{n, S,r}^{(s, \la)}$, with coefficients independent of $y$ and $v$.   
\end{lemma}

\begin{proof} 
We have
$(A \odot B)( x y^* \odot t(\tau_\pi \xbold^b)) = (A\, x y^* \odot B\,  t(\tau_\pi \xbold^b)),$
because $ t(\tau_\pi \xbold^b)$ has only vertical strands.  Therefore, also
$(A \odot B)( x y^* \odot t(c_{u, v}^\la)) = (A\, x y^* \odot B\,  t(c_{u, v}^\la)).$

 Note that $A \,x$ is an affine $(0, n-s)$--tangle, and can be reduced using the 
 algorithm of the proof of Propositions  2.18 and 2.19
 in ~\cite{GH2} to 
 a linear combination of stratified, flagpole descending $(0, n-s)$--tangles  $x'$ without closed loops.  The process does not affect $y^*$.

If $B$ has rank strictly less than $s$,  then the product 
$(A \odot B)( x y^* \odot t(c_{u, v}^\la))$ is a linear combination of basis elements $T_D$ with fewer than $s$ vertical strands, so belongs to
 $\breve W_{n, S, r}^{(s, \la)}$.  
 
 Otherwise, we can suppose that $B = g_\sigma \xbold^b$.  Then  $B\, t(c_{u, v}^\la) = t(\tau_\sigma \tbold^b) t(c_{u, v}^\la) \equiv
 t( \tau_\sigma \tbold^b \, c_{u, v}^\la)$  modulo the span of basis diagrams with fewer than $s$ vertical strands.   Moreover,  $t( \tau_\sigma \tbold^b  \, c_{u, v}^\la)$  is a linear combination of
 elements $t(  c_{u', v}^\la)$,  modulo $t(\breve H_{s, S, r}^\la)$, with coefficients independent of $v$, by the cellularity of the basis $\mathcal C_s$  of $H_{s, S, r}$.
 
 The conclusion follows from these observations.
 \end{proof}

\begin{theorem} \label{theorem:  cellular cyclotomic bmw}  Let $S$ be an admissible integral domain.
$(\mathcal C, \Lambda)$  is a cellular basis of the cyclotomic BMW algebra $\bmw {n, S,r}$.
\end{theorem}

\begin{proof}  Theorem \ref{theorem:  the x-i basis of cyclotomic BMW} and
 Lemma \ref{lemma:  basis of cyclotomic Hecke to basis of cyclotomic BMW}  implies that $\mathcal C$ is a basis of $\bmw {n, S,r}$, and  property (3) of cellular bases holds by Lemma \ref{lemma: * property of cellular basis}.  
 It remains to verify 
axiom (2) for cellular bases.  Thus we have to show that for $w \in \bmw {n, S,r}$, and for a basis element   $
c_{(\alpha, x, u), (\beta, y, v)}^{(s, \la)} =  g_\alpha ( x y^* \odot t(c_{u, v}^\la))  (g_\beta)^*
$,   the product
\begin{equation} \label{equation: task for cellularity proof}
w \  g_\alpha ( x y^* \odot t(c_{u, v}^\la))  (g_\beta)^*
\end{equation}
 can be written as a linear combination of elements
$$
g_{\alpha'} ( x' y^* \odot t(c_{u', v}^\la))  (g_\beta)^*,
$$
modulo $\breve W_{n, S,r}^{(s, \la)}$  (with coefficients independent of $(\beta, y, v)$).     

 It suffices to consider products as in  Equation (\ref{equation: task for cellularity proof})
 with $w$ equal to $e_i$ or to $g_i$  for some $i$,  or $w = x_1$.   We consider first $w = e_i$ or $w= g_i$.  Here there are several cases,  depending on the relative position of $\alpha\inv(i)$  and $\alpha\inv(i+1)$.  
 
Suppose that $\alpha\inv(i) > \alpha\inv(i+1)$.  Then $g_\alpha = g_i g_{\alpha_1}$, where
${\alpha_1}\inv(i) < \alpha_1\inv(i+1)$, and $\alpha_1$ is also a $(n-s, s)$-shuffle.
Thus $e_i g_\alpha = e_i g_i g_{\alpha_1} = \rho\inv e_i g_{\alpha_1}$.  Likewise,
$g_i g_\alpha = (g_i)^2 g_{\alpha_1} = g_{\alpha_1}  + (q - q\inv)  g_\alpha + (q\inv -q) \rho\inv e_i g_{\alpha_1}$.
We are therefore reduced to considering the case that $\alpha\inv(i) < 
\alpha\inv(i+1)$.

Suppose that $\alpha\inv(i+1) \le n-s $  or $n-s + 1 \le \alpha\inv(i)$.  Then  $\alpha\inv(i+1) = \alpha\inv(i) + 1$,  because $\alpha$ is a $(n-s, s)$--shuffle.  Write $\chi_i$ for $g_i$ or $e_i$.  We have 
$\chi_i g_\alpha =  g_\alpha \chi_{\alpha\inv(i)}$, as one can verify with pictures.
But $\chi_{\alpha\inv(i)} \in W_{n-s} \otimes W_s$,  so the conclusion follows from Lemma  \ref{lemma:  cellular cyclotomic BMW}. 
 
 It remains to examine the case that $\alpha\inv(i) \le n-s$  and $\alpha\inv(i+1) \ge n-s +1$.
In this case,  $g_i g_\alpha$ is  an $(n-s, s)$--shuffle  so
$g_i \  g_\alpha ( x y^* \odot t(c_{u, v}^\la))  (g_\beta)^*$
is another basis element.

Next, we have to consider the product
$e_i \  g_\alpha ( x y^* \odot t(c_{u, v}^\la))  (g_\beta)^*.$
 Define a permutation $\varrho$ by
\begin{equation} \label{equation:  bmw cellularity proof, defn of varrho}
\varrho(j) = 
\begin{cases}
j &\text{if  $j < \alpha\inv(i)$,}\cr
j+1  &\text{if $\alpha\inv(i) \le j < n-s$,}\cr
\alpha\inv(i) & \text{if $j = n-s$,}\cr
\alpha\inv(i+1) &\text{if $j = n-s + 1$,}\cr
j -1  &\text{if $n-s +1 < j \le \alpha\inv(i+1)$,}\cr
j & \text{if $ j > \alpha\inv(i+1)$.}\cr
\end{cases}
\end{equation}
Since $\varrho \in \S_{n-s} \times \S_r$, we have $\ell (\alpha \varrho) = \ell (\alpha) + 
\ell (\varrho)$ and
$g_{\alpha \varrho} = g_\alpha g_\varrho$.  The permutation $\alpha \varrho$ has the 
following properties:
$\alpha \varrho(n-s) = i$;  $\alpha \varrho(n-s+1) = i+1$;  if $1 \le a < b \le n-s-1$ or 
$n-s + 2 \le  a < b \le n$, then $\alpha\varrho(a) < \alpha\varrho(b)$.  We have
\begin{equation} \label{equation: bmw cellularity proof 1}
e_i  g_\alpha = e_i g_\alpha  g_\varrho g_\varrho\inv = e_i g_{\alpha \varrho} g_\varrho 
\inv.
\end{equation}
The tangle $e_i g_{\alpha \varrho}$ is stratified and has a horizontal strand connecting the top vertices $\p {n-s}$  and $\p {n-s + 1}$.  Contracting that strand,  we
get
\begin{equation} \label{equation: abmw cellularity proof 2}
e_i g_{\alpha \varrho} =   \cap_i \, g_\sigma \, \cup_{n-s},
\end{equation}
for a certain $(n-s-1, s-1)$--shuffle $\sigma$.
Therefore,
 \begin{equation}
  e_i\  g_\alpha ( x y^* \odot t(c_{u, v}^\la))  (g_\beta)^* = 
   \cap_i \,g_{\sigma} \cup_{n-s} g_\varrho \inv   ( x y^* \odot t(c_{u, v}^\la))  (g_\beta)^*.
 \end{equation}
 Moreover, 
 $g_\varrho\inv \in W_{n-s} \otimes W_{s} \subseteq \bmw {n, S, r}$,  so $g_\varrho \inv   ( x y^* \odot t(c_{u, v}^\la))  (g_\beta)^* $  is congruent modulo $\breve W_{n, S,r}^{(s, \la)}$ to a  linear combination of  elements
$ ( x' y^* \odot t(c_{u', v}^\la))   (g_\beta)^* $,  with coefficients independent of  $\beta$, $y$ and $v$,  by Lemma \ref{lemma:  cellular cyclotomic BMW}.  Thus we have to consider the products
\begin{equation} \label{equation: cyclotomic cellular proof1}
 \cap_i\, g_{\sigma} \cup_{n-s}  ( x' y^* \odot t(c_{u', v}^\la))  (g_\beta)^*.
\end{equation}

\begin{figure}
$$
\begin{aligned}
\inlinegraphic[scale = .9]{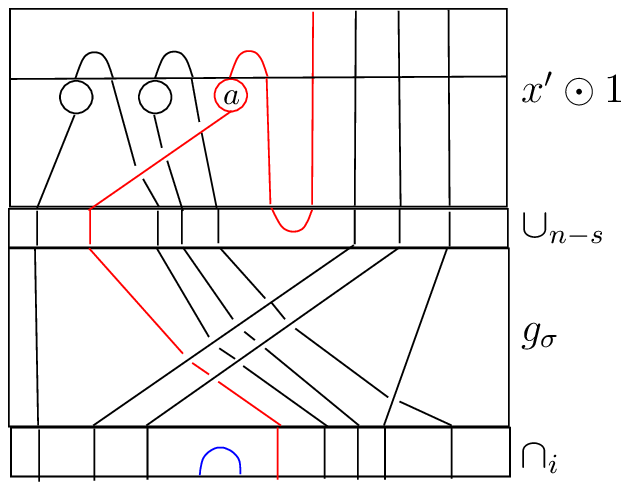} &= \inlinegraphic[scale=.9]{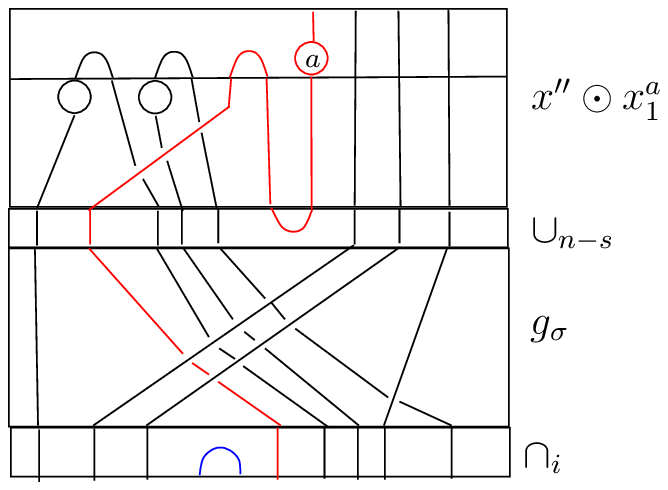}\\
&=  \inlinegraphic[scale=.9]{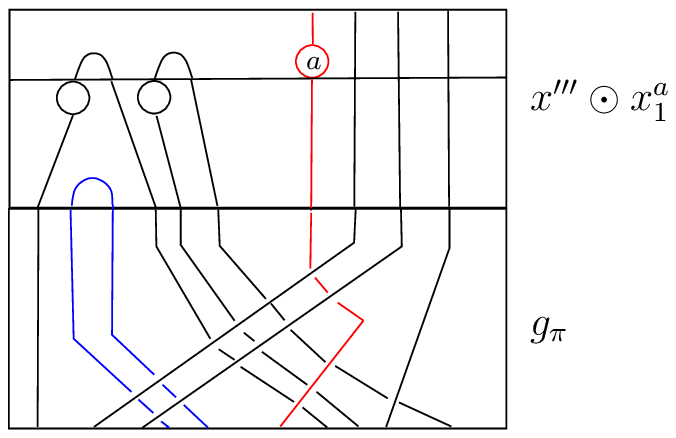}
\end{aligned}
$$
\caption{} \label{figure: affine cellular proof}
\end{figure}

Focus for a moment on the product $\cap_i\, g_{\sigma} \cup_{n-s}  ( x'  \odot 1)$, 
and write $x'$  in the form
$$
 g_{\alpha_2}  x_1^{a_1} x_3^{a_3} \cdots x_{n-s-1}^{a_{n-s-1}}( \cap_{n-s-1} \cdots  \cap_3  \cap_1), 
$$
with $\alpha_2 \in \enyang f f$  ($f = (n-s)/2$).    Figure \ref{figure: affine cellular proof} provides a guide to the computations.  Write $a = a_{n-s-1}$.  We have
\begin{equation}
\begin{aligned}
 \cap_i\, g_{\sigma} \cup_{n-s}  & ( x'  \odot 1)  \\
&=  \cap_i\, g_{\sigma}\cup_{n-s}  (g_{\alpha_2}  x_1^{a_1} x_3^{a_3} \cdots x_{n-s-1}^{a})( \cap_{n-s-1} \cdots  \cap_3  \cap_1) \\
&= \cap_i\, g_{\sigma} (g_{\alpha_2}  x_1^{a_1} x_3^{a_3} \cdots   x_{n-s-1}^{a})  \cup_{n-s}  \cap_{n-s-1} (  \cap_{n-s-3} \cdots  \cap_3  \cap_1).
\end{aligned}
\end{equation}
Note that
$$
\begin{aligned}
x_{n-s-1}^a \cup_{n-s}  \cap_{n-s-1} &= \rho^{-a} \cup_{n-s}\, x_{n-s}^{-a}\,\cap_{n-s-1} \\
& =  \cup_{n-s}\,\cap_{n-s-1}  x_{n-s+1}^{a}\, ,
\end{aligned}
$$
by  ~\cite{GH1},  Lemma 6.8 and Remark 6.9.  Thus,
\begin{equation}
\begin{aligned}
 \cap_i\, g_{\sigma}\cup_{n-s}  & ( x'  \odot 1)  \\
&=  \cap_i\, g_{\sigma} \cup_{n-s}  (g_{\alpha_2}  x_1^{a_1} x_3^{a_3} \cdots  x_{n-s-3}^{a_{n-s-3}} )   (\cap_{n-s-1}   \cdots  \cap_3  \cap_1)  x_{n-s+1}^{a} \\
&=   \cap_i\, g_{\sigma}\cup_{n-s} (x'' \odot x_1^a)
\end{aligned}
\end{equation}
where $x''$ is another stratified, flagpole descending affine $(0, n-s)$--tangle diagram
(without closed loops)  with the property that the strand
incident with the  vertex $\pbar {n-s}$  has winding number $0$ with the flagpole.
See the second stage in Figure \ref{figure: affine cellular proof}.   (In the figure,  a ``bead" on the $j$--th strand is supposed to indicate a power of $x_j$;  a bead labelled by $a$ indicates $x_j^a$.)

Since this affine  tangle diagram is stratified,  the strand incident with the top vertex
$\p {1}$  can be pulled straight,   and the horizontal strand connecting the bottom vertices $\pbar i$ and $\pbar {i+1}$  can be pulled up.   The result is an affine tangle diagram with the factorization $g_\pi  (x''' \odot x_1^a)$,  where $x'''$ is a stratified, flagpole descending affine $(0, n-s)$--tangle diagram and  $g_\pi$  is a positive permutation braid; this is illustrated in the final stage of Figure   \ref{figure: affine cellular proof}.

Consequently, we have:
\begin{equation} \label{equation: cyclotomic cellular proof2}
\begin{aligned}
 \cap_i\, &g_{\sigma} \cup_{n-s}  ( x' y^* \odot t(c_{u', v}^\la))  (g_\beta)^* \\&= 
g_\pi  ( x''' y^* \odot  x_1^a t(c_{u', v}^\la))  (g_\beta)^* \end{aligned}
\end{equation}
By Lemma \ref{lemma:  cellular cyclotomic BMW},  this is congruent mod
$\breve W_{n, S,r}^{(s, \la)}$  to a linear combination of terms
\begin{equation} \label{equation: cyclotomic cellular proof4}
\begin{aligned}
g_\pi  ( x''' y^* \odot   t(c_{u'', v}^\la))  (g_\beta)^*,
 \end{aligned}
\end{equation}
with coefficients independent of $y$, $\beta$, and $v$.

Finally, the permutation $\pi$  can be factored as $\pi = \pi_1 \pi_2$,   where
$\pi_1$ is an \break $(n-s, s)$--shuffle,  $\pi_2 \in \S_{n-s} \times \S_s$  and 
$\ell(\pi) = \ell(\pi_1) + \ell(\pi_2)$.  Consequently,   $g_\pi = g_{\pi_1}  g_{\pi_2}$, 
where  $g_{\pi_2} \in W_{n-s} \otimes W_s$.   We can now apply\ ÊLemma  \ref{lemma:  cellular cyclotomic BMW} again
to rewrite the product of Equation (\ref{equation: cyclotomic cellular proof4}) as a linear combination of elements \break
$
 g_{\pi_1}  ( x'''' y^* \odot t(c_{u''', v}^\la))  (g_\beta)^*, 
 $
modulo $\breve W_{n, S, r}^{(s, \la)}$, with coefficients independent of $\beta$, $y$, and $v$.  This completes the proof of the case: $w = e_i$,  $\alpha\inv(i) \le n-s$  and $n - s + 1 \le \alpha\inv(i+1) $. 

It remains to  consider the product
$$x_1 \  g_\alpha ( x y^* \odot t(c_{u, v}^\la))  (g_\beta)^*.$$
 Since $\alpha$ is an $(n-s, s)$--shuffle,  either $\alpha\inv(1) = 1$  or 
$\alpha\inv(1) = n-s + 1$.   In case $\alpha\inv(1) = 1$,  we have $x_1 g_\alpha = g_\alpha x_1$, and the result follows by applying  Lemma \ref{lemma:  cellular cyclotomic BMW}.  If $\alpha\inv(1) = n-s + 1$,  then we can write $g_\alpha$ as
$g_\alpha =  g_{\alpha_2} (g_1 g_2 \dots g_{n-s -1})  $, where 
$g_{\alpha_2}$ is a word in $g_j$, $j \ge 2$.  In this case,  
$x_1 g_\alpha =  g_{\alpha_2}(g_1\inv g_2\inv \dots g_{n-s -1}\inv)   x_{n-s+1}$. Now the result follows by first applying Lemma \ref{lemma:  cellular cyclotomic BMW}, and then
expanding each $g_j\inv$ in terms of $g_j$ and $e_j$  and appealing to the previous part of the proof.
\end{proof}

\bibliographystyle{amsplain}
\bibliography{cellular}

\end{document}